\documentclass[10pt,letterpaper,reqno]{amsart}
\usepackage[T1]{fontenc}
\usepackage{amsmath,amssymb,amsthm,mathtools}
\usepackage{xcolor}
\usepackage{hyperref}
\usepackage{doi}
\hypersetup{colorlinks=true,linkcolor=blue,citecolor=blue,urlcolor=blue}
\newtheorem{thm}{Theorem}
\newtheorem{lem}[thm]{Lemma}
\newtheorem{claim}[thm]{Claim}
\newtheorem{cor}[thm]{Corollary}
\DeclareMathOperator{\spanop}{span}
\DeclareMathOperator{\dist}{dist}

\begin{document}

\title[Kadison's problem 13 in the separable case]
{The separable case of Kadison's problem on orthonormal bases of unitaries for type $\mathrm{II}_1$ factors}

    \author[Y.~He]{Yixin He}
\address{School of Mathematical Sciences, Fudan University, Shanghai 200433, P.~R.~China}
\email{yixin.he717@gmail.com}

\author[Q.~Tang]{Quanyu Tang}
\address{School of Mathematics and Statistics, Xi'an Jiaotong University, Xi'an 710049, P.~R.~China}
\email{tangquanyu827@gmail.com}
	
	\author[T.~Zhang]{Teng Zhang}
	
	\address{School of Mathematics and Statistics, Xi'an Jiaotong University, Xi'an 710049, P.~R.~China}
	\email{teng.zhang@stu.xjtu.edu.cn}

    \thanks{The third author is supported by the China Scholarship Council, the Young Elite Scientists Sponsorship Program for PhD Students (China Association for Science and Technology), and the Fundamental Research Funds for the Central Universities at Xi'an Jiaotong University (Grant No.~xzy022024045). }

    \begin{abstract}  In 1967, Kadison asked ``does every type $\mathrm{II}_1$  factor have an orthonormal (with respect to the trace) basis consisting of unitaries?'' Using a noncommutative Lyapunov theorem of Akemann and Weaver, we prove that if \(M\)
is a  separable diffuse finite von Neumann algebra with a normal faithful trace
\(\tau\), then \(L^2(M,\tau)\) admits an
orthonormal basis consisting of self-adjoint unitaries in \(M\).
		 Consequently, we affirm the separable case of the Kadison problem.
	\end{abstract}
    
\maketitle
In 1967, at the Baton Rouge conference~\cite{Kad67}, Kadison asked the following trace-vector orthonormal  basis problem for \(\mathrm{II}_1\) factors\footnote{See also 
\cite[p.~622, Problem~13]{Ge03} or \cite[Problem~S.2]{Pet20}.}:

\par\vspace{0.03in}
\emph{Let \(M\) be a type \(\mathrm{II}_1\) factor with a normal faithful
trace \(\tau\). Does 
\(L^2(M,\tau)\) admit a complete orthonormal basis consisting of trace vectors?}
\par\vspace{0.03in}

This problem has an affirmative solution in the context of group von Neumann algebras of countable discrete groups by construction and finite von Neumann algebras that arise from group measure space construction; for example, see \cite{Chi73,Cho87}.  Beyond such examples, progress on this problem has been rather limited.
Recently, De and Mukherjee~\cite{DM23} proved that every separable diffuse
finite von Neumann algebra admits a uniformly bounded self-adjoint orthonormal
basis in its GNS space.  Their construction, however, does not yield unitaries.

The purpose of this note is to prove the following result, which answers the separable case of Kadison's problem.
\begin{thm}\label{thm:main}
Let $M$ be a separable diffuse finite von Neumann algebra with a normal faithful trace $\tau$. Then $L^2(M,\tau)$ has an orthonormal basis entirely consisting of self-adjoint unitaries in $M$.
\end{thm}

We use the following finite form of the Akemann--Weaver noncommutative Lyapunov theorem \cite[Theorem~3.1]{AW03}.   If $M$ is a diffuse von Neumann algebra, $\varphi_1,\ldots,\varphi_n\in M_*$ are self-adjoint, $\Phi(x)=(\varphi_1(x),\ldots,\varphi_n(x))$ on $M_{\rm sa}$, and $K=\{x\in M_{\rm sa}:-1\le x\le1\}$, then
\[
        \Phi(K)=\Phi(\mathcal S(M)),
\]
where $\mathcal S(M)=\{u\in M:u=u^*=u^{-1}\}$ is the set of self-adjoint unitaries. Indeed, for any \(x\in K\), put \(q=(1+x)/2\).  Applying
\cite[Theorem~3.1]{AW03} to the positive contraction \(q\), with no
constraints and with \(g_i=\varphi_i\), gives a projection \(p\in M\) such
that \(\varphi_i(p)=\varphi_i(q)\) for all \(i\); hence
\(u=2p-1\in\mathcal S(M)\) satisfies \(\Phi(u)=\Phi(x)\). The reverse inclusion is immediate from \(\mathcal S(M)\subset K\).

\begin{lem}\label{lem:norming}
Let M be a diffuse finite von Neumann algebra with  a normal faithful trace $\tau$. Let $F\subset M_{\rm sa}$ be a finite-dimensional real subspace and let $0\ne a\in M_{\rm sa}\cap F^\perp$.  Then there exists a self-adjoint unitary $u\in M$ such that $u\perp F$ and
$
\tau(au)=\|a\|_2^2/\|a\|_\infty.
$
\end{lem}

\begin{proof}
Choose an orthonormal basis $e_1,\ldots,e_r$ of $F$ and define
$
   \Phi(x)=\bigl(\tau(xe_1),\ldots,\tau(xe_r),\tau(xa)\bigr),  x\in M_{\rm sa}.
$
Since $a\perp F$,
$
   \Phi(a/\|a\|_\infty)=\bigl(0,\ldots,0,\|a\|_2^2/\|a\|_\infty\bigr).
$
The finite Akemann--Weaver theorem gives a self-adjoint unitary $u$ with this same image under $\Phi$.  Hence $u\perp F$ and $\tau(au)=\|a\|_2^2/\|a\|_\infty$.
\end{proof}
\begin{claim}\label{claim:approx}
Let \(M\) be a diffuse finite von Neumann algebra with a normal faithful
trace \(\tau\).  Let \(F\subset M_{\rm sa}\) be a finite-dimensional
real subspace, let \(a\in M_{\rm sa}\), and let \(\varepsilon>0\).  Then
there exist pairwise orthogonal self-adjoint unitaries
\(u_1,\ldots,u_n\in F^\perp\cap M\) such that, for
\(F'=\spanop_{\mathbb R}(F,u_1,\ldots,u_n)\),
$
        \|P^\perp_{F'}a\|_2<\varepsilon .
$
\end{claim}

\begin{proof}
Set \(a_0=P_F^\perp a\).  If \(\|a_0\|_2<\varepsilon\), there is nothing to
prove.  Otherwise, for $k\ge 1$, as long as \(a_{k-1}\ne0\), Lemma~\ref{lem:norming}
gives a self-adjoint unitary
$
   u_k\in \spanop_{\mathbb R}(F,u_1,\ldots,u_{k-1})^\perp
$
such that
$
   \alpha_k:=\tau(a_{k-1}u_k)
   =
   \|a_{k-1}\|_2^2/\|a_{k-1}\|_\infty.
$
  Put
$
   a_k=a_{k-1}-\alpha_k u_k.
$
Then
$
   a_k=P^\perp_{\spanop_{\mathbb R}(F,u_1,\ldots,u_k)}a
$
and orthogonality gives
$
   \|a_k\|_2^2=\|a_{k-1}\|_2^2-\alpha_k^2.
$
In particular,
$
   \sum_{j<k}\alpha_j^2\leq \|a_0\|_2^2.
$
Moreover,
$
   a_{k-1}=a_0-\sum_{j<k}\alpha_j u_j,
$
and hence
\[
   \|a_{k-1}\|_\infty
   \leq
   \|a_0\|_\infty+\sum_{j<k}|\alpha_j|
   \leq
   \|a_0\|_\infty+\sqrt{k-1}\,\|a_0\|_2 .
\]
If \(\|a_k\|_2\geq\varepsilon\) for all \(k\), then
\[
   \alpha_k^2
   \geq
   \frac{\varepsilon^4}
   {(\|a_0\|_\infty+\sqrt{k-1}\,\|a_0\|_2)^2}.
\]
The series on the right diverges, while
$
   \sum_k\alpha_k^2\leq \|a_0\|_2^2,
$
a contradiction.  Therefore \(\|a_n\|_2<\varepsilon\) for some \(n\), which
proves the claim.
\end{proof}
Now, a standard diagonalization argument gives Theorem~\ref{thm:main}.
\begin{proof}[Proof of Theorem~\ref{thm:main}]
Let $H_{\mathbb R}=L^2(M,\tau)_{\rm sa}$ and choose a $\|\cdot\|_2$-dense sequence $(b_j)$ in $M_{\rm sa}$.  Enumerate $\mathbb N^2$ as $(j_m,k_m)_{m\ge1}$.  Start with $E_0=\spanop_{\mathbb R}\{1\}$.  At stage $m$, apply Claim~\ref{claim:approx} to $F=E_{m-1}$, $a=b_{j_m}$, and $\varepsilon=2^{-k_m}$, and enlarge $E_{m-1}$ by the finitely many self-adjoint unitaries obtained; call the new space $E_m$.

The chosen self-adjoint unitaries, together with $1$, are mutually orthonormal.  For each fixed $j,k$, the pair $(j,k)$ occurs at some stage, so $\dist(b_j,E_m)<2^{-k}$ thereafter.  Hence every $b_j$ belongs to the closed real span of the chosen self-adjoint unitaries; by density this span is $H_{\mathbb R}$.  Since the complex span contains $H_{\mathbb R}+iH_{\mathbb R}\supset M$ and $M$ is dense in $L^2(M,\tau)$, the same family is an orthonormal basis of $L^2(M,\tau)$.
\end{proof}

\begin{cor}
If $M$ is a separable type $\mathrm{II}_1$ factor with a normal faithful
trace \(\tau\), then $L^2(M,\tau)$ has an orthonormal basis of trace vectors, represented by  self-adjoint unitaries.
\end{cor}

\begin{proof}
Such an $M$ is diffuse and $L^2(M,\tau)$ is separable.  Theorem~\ref{thm:main} gives a basis $(s_n)$ of self-adjoint unitaries.  For every $x\in M$,
\[
  \langle x\widehat{s_n},\widehat{s_n}\rangle_2=\tau(s_n^*xs_n)=\tau(x),
\]
so each $\widehat{s_n}$ is a trace vector.
\end{proof}
\noindent\emph{Acknowledgements.}
We thank Professor Liming Ge for pointing out to us that the separable case of Kadison's problem on orthonormal bases of unitaries for type $\mathrm{II}_1$ factors is a particularly important case. We are also deeply grateful to Professor Jesse Peterson for his careful reading of an early draft of this article, and for suggesting that we shorten the paper to no more than 4--6 pages in order to increase its impact. Some statements of this article are directly based on the beautiful one-page simplified proof that he sent us; we thank him for kindly allowing us to use it freely. We also thank Dr. Changying Ding for sharing with us an alternative
simplified proof of the finite form of the Akemann--Weaver theorem.

\medskip
\noindent\emph{Author Contributions.} The reference \cite{AW03} for the Akemann--Weaver noncommutative Lyapunov theorem was identified by the first two authors. The main ideas of the proof are due to the third author. The manuscript was written by the third author, based on Professor Jesse Peterson’s one-page simplified proof.

\end{document}